\documentclass{amsart}
\usepackage[utf8]{inputenc}
\usepackage{amssymb}
\usepackage[hidelinks]{hyperref}
\allowdisplaybreaks
\numberwithin{equation}{section}

\newtheorem{theorem}{Theorem}[section]
\newtheorem{lemma}[theorem]{Lemma}
\newtheorem{proposition}[theorem]{Proposition}
\newtheorem{corollary}[theorem]{Corollary}
\theoremstyle{remark}
\newtheorem{remark}[theorem]{Remark}
\newtheorem*{remark*}{Remark}
\newcommand{\R}{\mathbb R}
\newcommand{\eps}{\varepsilon}
\newcommand{\Om}{\Omega}
\newcommand{\dd}{\,\mathrm d}
\DeclareMathOperator{\diam}{diam}
\DeclareMathOperator{\dist}{dist}

\begin{document}
\title[Neumann eigenvalue comparison on thin convex domains]
{A quadratic comparison of Neumann eigenvalues
on thin convex domains in arbitrary dimension}
\author{Qixuan Hu}
\address{Department of Mathematics and Computer Sciences,
Shantou University, Guangdong, P. R. China}
\email{qxhu@stu.edu.cn}
\date{\today}
\thanks{This work was supported by SRIG under Grant NTF25026T}

\begin{abstract}
Let $\Omega\subset\mathbb R^n$ be a bounded convex domain that is thin around a
chosen diameter segment. We compare its Neumann spectrum with the spectrum
of that segment weighted by the $(n-1)$-dimensional volumes of its
perpendicular sections. We prove an $O(\varepsilon^2)$ comparison of the mean-zero
inverse operators and, consequently, an $O(\varepsilon^2)$ eigenvalue comparison
for every fixed index in every dimension $n\ge2$. The constants depend
only on the dimension and the eigenvalue index. Thin rectangles show that the quadratic exponent is optimal.
\end{abstract}
\maketitle
\tableofcontents

\section{Introduction}

The low-lying Neumann spectrum of a thin domain describes diffusion with
reflecting boundary conditions and vibrations subject to a vanishing
normal derivative. When a domain is much thinner in its transverse
directions than in its longitudinal direction, transverse oscillations
are energetically expensive. This suggests that low-energy states should
be approximately constant on transverse sections and should be described
by a one-dimensional problem. Quantifying this approximation is useful
both for reducing a multidimensional spectral problem to an ordinary
differential equation and for understanding how geometric degeneration
affects eigenvalues. Thin-domain spectral problems belong to a broader
theory of dimension reduction; Grieser's survey~\cite{Grieser} provides
an overview of its contexts, methods, and results.

For a convex domain with a chosen diameter segment, the appropriate
one-dimensional model retains the volumes of the perpendicular sections.
Writing these volumes as $h(x)$, the effective operator is formally
$-h^{-1}(hv')'$, with the natural Neumann conditions associated with its
weighted energy form. Indeed, lifting a function by $Jv(x,y)=v(x)$ gives
exactly the quotient
\[
 \frac{\int_{\Om}|\nabla Jv|^2}{\int_{\Om}|Jv|^2}
 =\frac{\int h|v'|^2\dd x}{\int h|v|^2\dd x}.
\]
The weight therefore records geometry that an unweighted interval would
lose. The min--max principle immediately gives
$\mu_k(\Om)\leq\mu_k(N)$, where $N$ denotes the weighted interval.
The main question is how accurately the reverse inequality holds,
uniformly over thin convex domains whose sectional shapes may vary or
degenerate.

One-dimensional approximation has also played an important role in
locating the first Neumann nodal set. Jerison~\cite{Jerison} studied the
first nodal line and the first positive Neumann eigenvalue of convex
planar domains. Choi, Jerison, and Kim~\cite{CJK} developed the
higher-dimensional case of this approach. These works motivate
the use of a longitudinal model while also illustrating the geometric
difficulties caused by moving transverse sections. 

A related direction asks for quantitative improvements of the
classical diameter lower bound. For a bounded convex domain
$\Omega\subset\mathbb R^n$, write $D=\operatorname{diam}(\Omega)$ and
define its projective width by
\[
 W(\Omega)=\inf_{e\in\mathbb S^{n-1}}
   \operatorname{diam}\bigl(\pi_{e^\perp}\Omega\bigr),
\]
where $\pi_{e^\perp}$ denotes orthogonal projection onto the hyperplane
perpendicular to $e$. In dimension two, this is the usual minimum
width, namely the minimum distance between two parallel supporting
lines. For planar convex domains it also coincides with the sectional
width used by Wang and Xu; see \cite[Lemma~8.1]{WX}. In higher dimensions,
$W(\Omega)$ measures transverse size relative to a line, rather than
the minimum distance between two parallel supporting hyperplanes.

\begin{theorem}[Wang--Xu, {\cite[Theorem~1.2]{WX}}, rescaled]
\label{thm:WX}
There exist universal constants $c_0,c_1>0$ such that every bounded
open convex planar domain $\Om$ of diameter one with
$0<W(\Om)<c_0$ satisfies
\[
 \mu_1(\Om)\ge\pi^2+c_1W(\Om)^2.
\]
\end{theorem}

Amato, Bucur, and Fragal\`a obtained a higher-dimensional quantitative
inequality in terms of the John ellipsoid, that is, the ellipsoid of
maximal volume contained in the convex domain. Write
$a_1\ge\cdots\ge a_n$ for its semi-axis lengths.

\begin{theorem}[Amato--Bucur--Fragal\`a, {\cite[Theorem~2]{ABF}}]
\label{thm:ABF}
For every $n\ge2$ there exists $c_n>0$ such that, if $\Om\subset\R^n$
is a bounded open convex domain of diameter $D$ and its John ellipsoid
has semi-axes $a_1\ge\cdots\ge a_n$, then
\[
 \mu_1(\Om)\ge\frac{\pi^2}{D^2}+c_n\frac{a_2^2}{D^4}.
\]
\end{theorem}

Their subsequent result also permits power-concave weights and nonlinear
energies. Write
\[
 \nu_p(\Om,\omega)=
 \inf_{\substack{0\ne u\in W^{1,p}(\Om,\omega)\\
                  \int_\Om\omega|u|^{p-2}u=0}}
 \frac{\int_\Om\omega|\nabla u|^p}{\int_\Om\omega|u|^p}.
\]
Here $W^{1,p}(\Om,\omega)$ is the weighted Sobolev space with finite
weighted $p$-mass and gradient energy; in particular,
$\nu_2(\Om,1)=\mu_1(\Om)$.

\begin{theorem}[Amato--Bucur--Fragal\`a, {\cite[Theorem~1]{ABFnonlinear}}]
\label{thm:ABFweighted}
Let $1<p<\infty$, $n\ge2$, and $r\ge1$ be an integer. There exists
$c(p,n,r)>0$ such that, for every bounded open convex domain
$\Om\subset\R^n$ of diameter $D$, and every positive integrable
weight $\omega$ for which $\omega^{1/r}$ is concave,
\[
 \nu_p(\Om,\omega)
 \ge\left(\frac{\pi_p}{D}\right)^p
       +c(p,n,r)\frac{a_2^2}{D^{p+2}},
 \qquad
 \pi_p=\frac{2\pi(p-1)^{1/p}}{p\sin(\pi/p)},
\]
where $a_2$ is the second largest semi-axis of the John ellipsoid.
\end{theorem}

These theorems measure excess above a diameter-based lower bound.
Our reference value is instead the spectrum of the weighted interval
formed from the actual sections perpendicular to a diameter.
For this comparison, the author's earlier work~\cite{Hu} establishes
the following two estimates.

\begin{theorem}[Hu, {\cite[Theorem~1.10]{Hu}}]
\label{thm:HuLinear}
For a bounded open convex domain $\Om\subset\R^n$, $n\ge2$, with  $\diam\Om=1$, and $N$ carries
the sectional-volume measure along the chosen diameter of $\Om$. Then, for
every positive integer $k$,
\[
 \bigl(1-2W(\Om)(1+\mu_k(N))\bigr)\mu_k(N)
 \le\mu_k(\Om)\le\mu_k(N).
\]
\end{theorem}

\begin{theorem}[Hu, {\cite[Theorem~1.11]{Hu}}]
\label{thm:HuQuadratic}
There exists a universal constant $C>0$ such that every bounded
open convex planar domain $\Om$ with $\diam\Om=1$, and $N$ carries
the sectional-volume measure along the chosen diameter of $\Om$ satisfies
\[
 \mu_1(N)-CW(\Om)^2\le\mu_1(\Om)\le\mu_1(N).
\]
\end{theorem}

For the case of first Neumann eigenvalue, i.e. $k=1$, Theorem~\ref{thm:HuLinear} is due to \cite{Jer} for $n=2$, and \cite{CJK} for $n\geq2$.
The proof in {\cite{Hu} is different and applies to every eigenvalue index.

Theorem~\ref{thm:HuLinear} compares all indices, whereas
Theorem~\ref{thm:HuQuadratic} improves the error to quadratic order
for the first positive planar eigenvalue. A natural question is
whether this quadratic comparison extends to every fixed positive
index in arbitrary dimension.

We answer this question by comparing the Neumann Laplacian with a
weighted operator on a diameter segment. Normalize
$\operatorname{diam}(\Omega)=1$ and choose coordinates so that the
endpoints of a diameter are $(0,0)$ and $(1,0)$. Suppose that
\[
 \Omega\subset(0,1)\times B_{\varepsilon}^{n-1}(0),
\]
where $B_{\varepsilon}^{n-1}(0)$ is the open ball of radius
$\varepsilon$ in $\mathbb R^{n-1}$. For $0<x<1$, set
\[
 \Omega_x=\{y\in\mathbb R^{n-1}:(x,y)\in\Omega\},
 \qquad h(x)=|\Omega_x|.
\]
Let $N$ denote the interval $(0,1)$ equipped with the measure
$h(x)\,dx$. Its Neumann operator is associated with the quadratic form
\[
 v\longmapsto\int_0^1 h(x)|v'(x)|^2\,dx
\]
on the weighted Sobolev space defined in Section~\ref{sec:statement}.
The endpoint conditions are understood in the natural variational
sense, including when $h$ vanishes at an endpoint. We index both
spectra from zero, counting multiplicities:
\[
 0=\mu_0(\Omega)<\mu_1(\Omega)\le\cdots,
 \qquad 0=\mu_0(N)<\mu_1(N)\le\cdots.
\]

Theorem~\ref{thm:main} establishes, for every $n\ge2$ and $k\ge1$,
\[
 \frac{\mu_k(N)}{1+16^n\varepsilon^2\mu_k(N)}
 \le\mu_k(\Omega)\le\mu_k(N).
\]
Consequently,
\[
 0\le\mu_k(N)-\mu_k(\Omega)
 \le16^n\varepsilon^2\mu_k(N)^2
 \le C_{n,k}\varepsilon^2,
\]
where $C_{n,k}$ depends only on the dimension and the index.
No lower bound on the inradii of the sections is required. In particular,
sections may themselves collapse in some transverse directions.

The tubular formulation can also be expressed in terms of projective
width. If $\ell$ is the chosen diameter line and
\[
 R_\ell=\sup_{z\in\Omega}\operatorname{dist}(z,\ell),
\]
then
\[
 R_\ell\le W(\Omega)\le2R_\ell;
\]
see the remark following Theorem~\ref{thm:main}. Consequently, for a
domain of arbitrary diameter $D$, the corresponding diameter-based
weighted interval satisfies
\[
 0\le\mu_k(N)-\mu_k(\Omega)
 \le C_{n,k}\frac{W(\Omega)^2}{D^4}.
\]
Here $N$ is constructed on the chosen diameter segment of length $D$.
Thin planar rectangles show that the exponent two is optimal for
every fixed positive index.

The proof compares the inverses of the two operators on their mean-zero
spaces. A transverse Poincar\'{e} estimate and a weighted estimate for
the motion of sections first yield an energy error of order $\eps$.
Exact energy orthogonality then expresses the quadratic form of the
inverse difference as the square of that error, producing order
$\eps^2$. After establishing the geometric and residul estimates, Section~\ref{sec:inverse} proves the inverse comparison,
and the min--max principle in Section~\ref{sec:comparison} gives the eigenvalue bounds.
Section~\ref{sec:rectsharp} gives thin planar rectangles for which
the gap divided by the squared width tends to $k^2\pi^2>0$,
proving optimality of the exponent.

\textbf{AI usage declaration}. The main content of this article was written by the author, following author's earlier work {\cite{Hu}. GPT-6 was used to assist with the eigenvalue calculations for a specific weighted segment in the section on sharpness and to improve the language and presentation. The author reviewed the AI-generated content and takes full responsibility for the final manuscript.

\section{Statement and normalization}\label{sec:statement}

Let $n\ge2$ and put $m=n-1$. Let $\Om\subset\R^n$ be a bounded open
convex set with nonempty interior and $\diam\Om=1$. Choose diameter
endpoints in $\overline\Om$ and use a rigid motion to write them as
\[
 A=(0,0),\qquad B=(1,0),\qquad (x,y)\in\R\times\R^m.
\]
Suppose that, for some $\eps>0$,
\begin{equation}\label{eq:tube}
 \Om\subset (0,1)\times B^m_\eps(0).
\end{equation}
Define the sections and their volumes by
\[
 \Omega_x=\{y\in\R^m:(x,y)\in\Om\},\qquad
 h(x)=|\Omega_x|,\qquad 0<x<1.
\]
The interval $N=[0,1]$ is equipped with the measure $h(x)\dd x$.
Its Neumann operator is defined through the quadratic form
\begin{align}
 \mathfrak a_N(v,\phi)&=\int_0^1hv'\phi'\dd x,\label{eq:formN}\\
 \mathcal V_h&=\left\{v\in W^{1,2}_{\mathrm{loc}}(0,1):
        \int_0^1h\bigl(|v|^2+|v'|^2\bigr)\dd x<\infty\right\}.
        \label{eq:formdomain}
\end{align}
Here $v'$ is the weak derivative. We use the 
continuous representative whenever pointwise values are needed.
The endpoint conditions are natural conditions for the form; no
classical endpoint derivative is assumed when $h$ vanishes.
We index the two spectra, counting multiplicities, by
\[
 0=\mu_0(\Om)<\mu_1(\Om)\le\mu_2(\Om)\le\cdots,
 \qquad
 0=\lambda_0<\lambda_1\le\lambda_2\le\cdots,
\]
where $\lambda_k=\mu_k(N)$.

\begin{theorem}\label{thm:main}
Under the assumptions above, for every integer $k\ge1$,
\begin{equation}\label{eq:main}
 \boxed{\quad
 \frac{\lambda_k}{1+16^n\eps^2\lambda_k}
 \le\mu_k(\Om)\le\lambda_k.
 \quad}
\end{equation}
In particular,
\begin{equation}\label{eq:gap}
 0\le\lambda_k-\mu_k(\Om)
 \le16^n\eps^2\lambda_k^2
 \le C_{n,k}\eps^2,
\end{equation}
where the explicit constant
\begin{equation}\label{eq:constantnk}
 C_{n,k}=256^n\pi^4(k+1)^4
\end{equation}
is not intended to be sharp.
\end{theorem}

\begin{remark*}[Relation to the width]
In \cite{Hu}, the projective width is defined by
\begin{equation}\label{eq:width}
 W(\Om)=\min_{e\in\mathbb S^{n-1}}
             \diam\bigl(\pi_{e^\perp}\Om\bigr).
\end{equation}
If $\ell$ is any diameter line and
$R_\ell=\sup_{X\in\Om}\dist(X,\ell)$, then
\[
 R_\ell\le W(\Om)\le2R_\ell.
\]
For the first inequality, the triangle formed by the diameter
endpoints and $X$ has minimum planar width $\dist(X,\ell)$:
its longest side is the diameter, and its shortest altitude is
the altitude from $X$. Every transverse projection of $\Om$ has
diameter at least this planar width, by projecting further onto
a direction in the triangle's plane. The second inequality
follows by projecting perpendicular to $\ell$.
Thus Theorem~\ref{thm:main} applies with $\eps=W(\Om)$. 
For a general diameter $D$, dilation gives
\[
 0\le\mu_k(N)-\mu_k(\Om)
 \le C_{n,k}\frac{W(\Om)^2}{D^4}.
\]
In $n\ge3$, this projective width controls all transverse
directions and is different from the width of a containing slab.
\end{remark*}

\section{Estimates on convex sections}

We first establish a Poincar\'{e} estimate with a dimension-dependent
constant by an elementary argument. This avoids any dependence on a
section's shape.

\begin{lemma}\label{lem:poincare}
Let $K\subset\R^m$ be a bounded open convex set with positive volume, and
let $\ell=\diam K$. For $q\in H^1(K)$, write
$q_K=|K|^{-1}\int_Kq$. Then
\begin{equation}\label{eq:poincareK}
 \|q-q_K\|_{L^2(K)}
 \le 2^{(m-1)/2}\ell\,\|\nabla q\|_{L^2(K)}.
\end{equation}
If, in addition, $0\in\overline K$ and $K\subset B^m_\eps(0)$, and
$\nu_K$ is the outward unit normal, put
\[
 n=m+1,\qquad p_n=2^{n/2},\qquad b_n=1+(n-1)p_n.
\]
Then for $q$ smooth near $\overline K$,
\begin{equation}\label{eq:weightedtrace}
 \int_{\partial K}(y\cdot\nu_K)|q-q_K|\dd\sigma
 \le b_n\eps\sqrt{|K|}\,\|\nabla q\|_{L^2(K)}.
\end{equation}
Here and below, when $m=1$, boundary integrals use counting measure at
the two interval endpoints.
\end{lemma}

\begin{proof}
For smooth $q$, the variance identity is
\[
 \int_K|q-q_K|^2
 =\frac1{2|K|}\int_K\int_K|q(y)-q(z)|^2\dd y\dd z.
\]
By the fundamental theorem of calculus, convexity, and
Cauchy--Schwarz,
\[
 |q(y)-q(z)|^2
 \le\ell^2\int_0^1|\nabla q((1-t)y+tz)|^2\dd t.
\]
For $0\le t\le1/2$, fix $z$ and make the change of variables
$w=(1-t)y+tz$. Since $(1-t)K+tz\subset K$,
\[
 \int_K\int_K|\nabla q((1-t)y+tz)|^2\dd y\dd z
 \le2^m|K|\int_K|\nabla q|^2.
\]
For $1/2\le t\le1$, exchange $y$ and $z$ to obtain the same estimate.
Integration in $t$ proves
\[
 \int_K|q-q_K|^2\le2^{m-1}\ell^2\int_K|\nabla q|^2.
\]
Density proves \eqref{eq:poincareK} for $H^1(K)$.

For the second assertion, convexity and $0\in\overline K$ imply
$y\cdot\nu_K\ge0$ almost everywhere on $\partial K$.
Apply the divergence theorem to the vector field
$y|q-q_K|$. The absolute value can be justified by a smooth
approximation. It gives
\begin{align*}
 \int_{\partial K}(y\cdot\nu_K)|q-q_K|\dd\sigma
 &=\int_K\left(m|q-q_K|+y\cdot\nabla|q-q_K|\right)\dd y\\
 &\le m\sqrt{|K|}\,\|q-q_K\|_2
       +\eps\sqrt{|K|}\,\|\nabla q\|_2.
\end{align*}
Since $\ell\le2\eps$, \eqref{eq:poincareK} gives
\[
 \|q-q_K\|_2\le2^{(m+1)/2}\eps\|\nabla q\|_2
             =p_n\eps\|\nabla q\|_2.
\]
Substitution yields \eqref{eq:weightedtrace}.
\end{proof}

\section{Geometry and differentiation of sectional averages}

The diameter condition implies that the projection of $\Om$ onto the
$x$-axis is $(0,1)$ and that the endpoint sections of $\overline\Om$
are the singletons $A$ and $B$. Each $\Omega_x$ is a bounded open convex set
of positive $m$-dimensional volume, and $0\in\overline{\Omega_x}$.
The latter fact follows from the inclusion $[A,B]\subset\overline\Om$.

Define
\begin{equation}\label{eq:JP}
 (J\phi)(x,y)=\phi(x),\qquad
 (Pz)(x)=\frac1{h(x)}\int_{\Omega_x}z(x,y)\dd y.
\end{equation}
Then $J:L^2(h\dd x)\to L^2(\Om)$ is an isometry,
$P=J^*$, and $PJ=I$. In particular, $JP$ is the orthogonal
projection onto the subspace of functions that are constant on each
section. Fubini's theorem and Lemma~\ref{lem:poincare} give
\begin{equation}\label{eq:transverseP}
 \|z-JPz\|_{L^2(\Om)}
 \le p_n\eps\|\nabla_yz\|_{L^2(\Om)},
 \qquad z\in H^1(\Om).
\end{equation}

Write $d(x)=\min\{x,1-x\}$. The next lemma is the geometric ingredient
that replaces the upper and lower boundary graphs in the planar proof in {\cite{Hu}}.

\begin{lemma}\label{lem:average}
For $z$ smooth on $\overline\Om$, the sectional
average $Pz$ is locally absolutely continuous on $(0,1)$. For almost
every $x$, it satisfies
\begin{equation}\label{eq:averageDerivative}
 (Pz)'(x)=P(z_x)(x)+B_z(x),
\end{equation}
where
\begin{equation}\label{eq:Bbound}
 |h(x)B_z(x)|
 \le\frac{b_n\eps}{d(x)}\sqrt{h(x)}
       \left(\int_{\Omega_x}|\nabla_yz|^2\dd y\right)^{1/2}.
\end{equation}
\end{lemma}

\begin{proof}
At almost every lateral boundary point $(x,y)$ of $\Om$, write the
outward unit normal as $\nu=(\nu_x,\nu_y)$. The supporting-plane
inequalities at the two endpoints give
\[
 -x\nu_x-y\cdot\nu_y\le0,
 \qquad (1-x)\nu_x-y\cdot\nu_y\le0.
\]
At a point with $0<x<1$, these inequalities rule out $\nu_y=0$.
For almost every such point, the outward normal of the slice is
\[
 \nu_{\Omega_x}=\frac{\nu_y}{|\nu_y|}.
\]
Set
\[
 p=y\cdot\nu_{\Omega_x},\qquad
 V=-\frac{\nu_x}{|\nu_y|}.
\]
The supporting-plane inequalities imply
\begin{equation}\label{eq:velocity}
 p\ge0,\qquad
 -\frac{p}{1-x}\le V\le\frac{p}{x},
 \qquad |V|\le\frac{p}{d(x)}.
\end{equation}

For clarity, the differentiation formulas can be obtained directly
in weak form, even when the boundary is not smooth. Put
\[
 M_z(x)=\int_{\Omega_x}z(x,y)\dd y.
\]
For $\psi\in C_c^1(0,1)$, apply the divergence theorem on $\Om$ to
$\psi(x)z(x,y)e_x$, where $e_x=(1,0)$ is the unit normal vector. The boundary coarea formula has tangential
Jacobian $|\nu_y|$. It follows that
\[
 \int_0^1\psi'M_z\dd x
 +\int_0^1\psi\int_{\Omega_x}z_x\dd y\dd x
 =-\int_0^1\psi\int_{\partial \Omega_x}Vz\dd\sigma\dd x.
\]
Consequently, in distributions on $(0,1)$,
\begin{equation}\label{eq:transport}
 M_z'=\int_{\Omega_x}z_x\dd y+
              \int_{\partial \Omega_x}Vz\dd\sigma,
 \qquad
 h'=\int_{\partial \Omega_x}V\dd\sigma.
\end{equation}
The second equality is obtained by letting $z=1$ be constant function on $\Omega$.

These distributional derivatives are locally integrable. Indeed,
\eqref{eq:velocity} and the divergence theorem in $\Omega_x$ give
\begin{equation}\label{eq:Vintegrable}
 \int_{\partial \Omega_x}|V|\dd\sigma
 \le\frac1{d(x)}\int_{\partial \Omega_x}p\dd\sigma
 =\frac{m h(x)}{d(x)}.
\end{equation}
Also $h$ is locally bounded below by a positive number. To see this,
fix one section of positive volume and use homotheties toward $A$
and $B$, which are contained in $\Om$ by convexity. Therefore
$M_z/h=Pz$ is locally absolutely continuous. Differentiating this
quotient in \eqref{eq:transport} gives
\begin{equation}\label{eq:Bformula}
 h B_z=\int_{\partial \Omega_x}V\bigl(z-Pz\bigr)\dd\sigma.
\end{equation}
Finally, apply \eqref{eq:velocity} and \eqref{eq:weightedtrace} to
$q(y)=z(x,y)$ on $\Omega_x$:
\[
 |h B_z|
 \le\frac1{d(x)}\int_{\partial \Omega_x}p|z-Pz|\dd\sigma
 \le\frac{b_n\eps}{d(x)}\sqrt h
                \|\nabla_yz(x,\cdot)\|_{L^2(\Omega_x)}.
\]
This is \eqref{eq:Bbound}.

\end{proof}

\section{The exsitence of inverse operator}

We first justify the inverse used below. Put
$M_h=\int_0^1h\dd x=|\Om|$. Convexity implies that $h$ is bounded
above and bounded away from zero on each compact subinterval of
$(0,1)$. Thus $\mathcal V_h$, where
\[
\mathcal V_h=\left\{v\in W^{1,2}_{\mathrm{loc}}(0,1):
\int_0^1h\bigl(|v|^2+|v'|^2\bigr)\dd x<\infty\right\}.
\]
 equipped with
\[
 \|v\|_{\mathcal V_h}^2=\int_0^1h(|v|^2+|v'|^2)\dd x,
\]
is a Hilbert space.

For $v\in\mathcal V_h$, local weak differentiation and Fubini's
theorem give $Jv\in H^1(\Om)$ and
\begin{equation}\label{eq:isometryH1}
 \|Jv\|_{H^1(\Om)}^2
 =\int_0^1h(|v|^2+|v'|^2)\dd x.
\end{equation}
Let
\[
 \mathcal H_N=\left\{f\in L^2(h\dd x):\int_0^1hf\dd x=0\right\},
 \qquad \mathcal V_h^0=\mathcal V_h\cap\mathcal H_N.
\]
Applying Lemma~\ref{lem:poincare} to $Jv$ on the diameter-one
domain $\Om$ yields
\begin{equation}\label{eq:weightedPoincare}
 \int_0^1h|v|^2\dd x\le2^{n-1}\int_0^1h|v'|^2\dd x,
 \qquad v\in\mathcal V_h^0.
\end{equation}
Hence $\mathfrak a_N$ is coercive on the closed subspace
$\mathcal V_h^0$, and its energy norm
$\|v\|_a^2=\int_0^1h|v'|^2\dd x$ is equivalent to the full form
norm there. For every $f\in\mathcal H_N$, Lax--Milgram therefore
gives a unique $v\in\mathcal V_h^0$ satisfying;
\[
	\int_0^1hv'\phi'\dd x=\int_0^1hf\phi\dd x,
	\qquad \phi\in\mathcal V_h.
\]
see \cite[Chapter~5]{Brezis}.
Testing with arbitrary $\phi\in\mathcal V_h$ is legitimate, since $\int_0^1hf=0$.

This defines the bounded solution operator $A_N^{-1}$ on
$\mathcal H_N$. It is positive, injective, and self-adjoint. Rellich compactness on the bounded
convex Lipschitz domain $\Om$, together with \eqref{eq:isometryH1},
gives the compact embedding
$\mathcal V_h\hookrightarrow L^2(h\dd x)$; see
\cite[Chapter~9]{Brezis}. Consequently $A_N^{-1}$ is compact.
The same argument on $H^1(\Om)\cap L^2_0(\Om)$ defines the
compact positive self-adjoint inverse $A_\Om^{-1}$.
All inverses here are restricted to the mean-zero spaces, since
constants form the Neumann kernel.

\begin{lemma}\label{lem:flux}
Let $f\in L^2(h\dd x)$ satisfy $\int_0^1hf\dd x=0$.
There exists a unique $v\in\mathcal V_h$ with
$\int_0^1hv\dd x=0$ solving
\begin{equation}\label{eq:weak1d}
 \int_0^1hv'\phi'\dd x=\int_0^1hf\phi\dd x,
 \qquad \phi\in\mathcal V_h.
\end{equation}
Then
\begin{equation}\label{eq:fluxIdentity}
 hv'=-F,\qquad
 F(x)=\int_0^xh(s)f(s)\dd s
      =-\int_x^1h(s)f(s)\dd s,
\end{equation}
and
\begin{equation}\label{eq:fluxEstimate}
 \int_0^1h\frac{|v'|^2}{d(x)^2}\dd x
 \le4^n\int_0^1h|f|^2\dd x.
\end{equation}
\end{lemma}

\begin{proof}
The existence and uniqueness were established above by Lax--Milgram.
For the endpoint argument, set $w=hv'$. Cauchy--Schwarz gives
\[
 \int_0^1|w|\dd x\le\sqrt{M_h}\|v\|_a<\infty,
 \qquad
 \int_0^1|hf|\dd x\le\sqrt{M_h}\|f\|_{L^2(h\dd x)}<\infty.
\]
Testing the weak equation with $C_c^\infty(0,1)$ functions gives
$w'=-hf$ in distributions. Consequently $w\in W^{1,1}(0,1)$ and
has well-defined traces $w(0),w(1)$ and an absolutely continuous
representative on $[0,1]$. For any $\phi\in C^\infty([0,1])$,
integration by parts yields
\[
 \int_0^1w\phi'\dd x
 =w(1)\phi(1)-w(0)\phi(0)+\int_0^1hf\phi\dd x.
\]
Comparing this identity with the weak equation shows that the
boundary term is zero. Choosing $\phi(x)=x$ gives $w(1)=0$;
choosing $\phi(x)=1-x$ gives $w(0)=0$. These smooth test functions
belong to $\mathcal V_h$ since $\int_0^1h<\infty$.
Thus
\[
 (hv')(0)=(hv')(1)=0.
\]
Integrating $w'=-hf$
from either endpoint proves \eqref{eq:fluxIdentity}.

We need a simple volume comparison. If $0<s<x<1$, convexity and
the endpoint $B$ give the inclusion
\[
 \frac{1-x}{1-s}\Omega_s\subset \Omega_x.
\]
Taking $m$-dimensional volumes yields
\begin{equation}\label{eq:volumecomp}
 h(s)\le\left(\frac{1-s}{1-x}\right)^m h(x).
\end{equation}
For $0<x\le1/2$, let $H(x)=\int_0^xh(s)\dd s$. Equation
\eqref{eq:volumecomp} implies
\begin{equation}\label{eq:massbound}
 H(x)\le2^m x h(x).
\end{equation}
Therefore
\begin{equation}\label{eq:HardyReduction}
 \int_0^{1/2}\frac{|F(x)|^2}{x^2h(x)}\dd x
 \le4^m\int_0^{1/2}\frac{|F(x)|^2h(x)}{H(x)^2}\dd x.
\end{equation}
Make the change of variables $t=H(x)$ and put
$q(t)=f(H^{-1}(t))$. The last integral becomes
\[
 \int_0^{H(1/2)}\left|\frac1t\int_0^tq(s)\dd s\right|^2\dd t.
\]
We use the elementary Hardy inequality
\begin{equation}\label{eq:Hardy}
 \int_0^L\left|\frac1t\int_0^tq(s)\dd s\right|^2\dd t
 \le4\int_0^L|q(t)|^2\dd t.
\end{equation}
For completeness, for a smooth real-valued $q$ write
$Q(t)=\int_0^tq(s)\dd s$. Integration by parts gives
\[
 \int_0^L\frac{Q^2}{t^2}\dd t
 =-\frac{Q(L)^2}{L}+2\int_0^L\frac{Qq}{t}\dd t
 \le2\left(\int_0^L\frac{Q^2}{t^2}\dd t\right)^{1/2}
          \|q\|_{L^2(0,L)}.
\]
This proves \eqref{eq:Hardy}; approximation extends it to $L^2$.
Combining \eqref{eq:HardyReduction} and \eqref{eq:Hardy}, we obtain
\[
 \int_0^{1/2}\frac{|F(x)|^2}{x^2h(x)}\dd x
 \le4^{m+1}\int_0^{1/2}h|f|^2\dd x.
\]
Reflecting $x\mapsto1-x$ gives the same bound on the right half,
using the tail representation of $F$ in \eqref{eq:fluxIdentity}.
Adding the two half-interval estimates and recalling $m+1=n$
proves \eqref{eq:fluxEstimate}.
\end{proof}

\section{Residual estimate for the lifted inverse problem}

Let
\[
 \mathfrak a_\Om(u,z)=\int_\Om\nabla u\cdot\nabla z.
\]
For $f$ and $v$ as in Lemma~\ref{lem:flux}, define the residual
\begin{equation}\label{eq:rdef}
 r_f(z)=\mathfrak a_\Om(Jv,z)-\langle Jf,z\rangle_{L^2(\Om)}.
\end{equation}

\begin{lemma}\label{lem:residual}
For every $z\in H^1(\Om)$,
\begin{equation}\label{eq:residual}
 |r_f(z)|\le2^n b_n\eps\,
       \|f\|_{L^2(h\dd x)}\|\nabla_yz\|_{L^2(\Om)}.
\end{equation}
Moreover,
\begin{equation}\label{eq:residualOrthogonal}
 r_f(J\phi)=0,\qquad \phi\in\mathcal V_h.
\end{equation}
\end{lemma}

\begin{proof}
First take $z$ smooth on a neighborhood of $\overline\Om$.
Lemma~\ref{lem:average} gives
\[
 \mathfrak a_\Om(Jv,z)
 =\int_0^1hv'P(z_x)\dd x
 =\int_0^1hv'(Pz)'\dd x-\int_0^1hv'B_z\dd x.
\]
The first term can be integrated by parts using
$hv'=-F$. In fact, $Pz$ is bounded and locally absolutely
continuous, and $F(0)=F(1)=0$. The terms are integrable by
\eqref{eq:Bbound} and \eqref{eq:fluxEstimate}; integrating first
on $[\delta,1-\delta]$ and then letting $\delta\downarrow0$
therefore gives
\begin{equation}\label{eq:residualFormula}
 r_f(z)=-\int_0^1hv'B_z\dd x.
\end{equation}
By \eqref{eq:Bbound} and Cauchy--Schwarz,
\begin{align*}
 |r_f(z)|
 &\le b_n\eps\int_0^1\frac{|v'|}{d(x)}\sqrt h
             \|\nabla_yz(x,\cdot)\|_{L^2(\Omega_x)}\dd x\\
 &\le b_n\eps
       \left(\int_0^1h\frac{|v'|^2}{d(x)^2}\dd x\right)^{1/2}
       \|\nabla_yz\|_{L^2(\Om)}.
\end{align*}
Lemma~\ref{lem:flux} proves \eqref{eq:residual}.
Density of restrictions of smooth functions in $H^1(\Om)$ then
proves the bound for all $z\in H^1(\Om)$.

Finally, \eqref{eq:weak1d} directly gives
\[
 \mathfrak a_\Om(Jv,J\phi)
 =\int_0^1hv'\phi'\dd x
 =\int_0^1hf\phi\dd x
 =\langle Jf,J\phi\rangle,
\]
which proves \eqref{eq:residualOrthogonal}.
\end{proof}

\section{Quadratic comparison of inverse operators}\label{sec:inverse}

Write
\[
 \mathcal H_\Om=\left\{g\in L^2(\Om):\int_\Om g=0\right\},
 \qquad
 \mathcal H_N=\left\{f\in L^2(h\dd x):\int_0^1hf\dd x=0\right\}.
\]
Both $P$ and $J$ preserve zero integral. Let $A_\Om$ and $A_N$
be the positive Neumann operators restricted to these mean-zero
spaces. Their inverse operators are well defined and compact by
coercivity and the compact embeddings discussed above. On
$\mathcal H_\Om$, set
\begin{equation}\label{eq:TS}
 T=A_\Om^{-1},\qquad S=JA_N^{-1}P.
\end{equation}
Both are compact positive self-adjoint operators.

\begin{theorem}\label{thm:inverse}
On $\mathcal H_\Om$, we have
\begin{equation}\label{eq:inverse}
 \boxed{\qquad 0\le T-S\le16^n\eps^2 I.\qquad}
\end{equation}
Where $I$ is the identity operator.
\end{theorem}

\begin{proof}
Take $g\in\mathcal H_\Om$ and put
\[
 f=Pg,\qquad u=Tg,\qquad v=A_N^{-1}f,\qquad
 w=u-Jv=(T-S)g.
\]
Since $P$ and $J$ preserve integrals, $f=Pg$ has zero weighted mean,
$u$ and $Jv$ have zero integral on $\Om$, and
$w\in H^1(\Om)\cap\mathcal H_\Om$.
For every $z\in H^1(\Om)\cap\mathcal H_\Om$, the defining equation
for $u=A_\Om^{-1}g$ gives
\[
 \mathfrak a_\Om(u,z)=\langle g,z\rangle_{L^2(\Om)}.
\]
On the other hand, the residual in \eqref{eq:rdef} is defined by
\[
 r_f(z)=\mathfrak a_\Om(Jv,z)-\langle Jf,z\rangle_{L^2(\Om)},
\]
so that
$\mathfrak a_\Om(Jv,z)=\langle Jf,z\rangle+r_f(z)$.
By linearity, subtraction gives
\begin{align*}
 \mathfrak a_\Om(w,z)
 &=\mathfrak a_\Om(u,z)-\mathfrak a_\Om(Jv,z)\\
 &=\langle g,z\rangle-\langle Jf,z\rangle-r_f(z).
\end{align*}
Thus
\begin{equation}\label{eq:wweak}
 \mathfrak a_\Om(w,z)
 =\langle g-Jf,z\rangle-r_f(z).
\end{equation}
All pairings here are in $L^2(\Om)$.

The projection identity used next follows directly
from Fubini's theorem:
\[
 \langle Jf,z\rangle_{L^2(\Om)}
 =\langle f,Pz\rangle_{L^2(h\dd x)}
 =\langle Pg,Pz\rangle_{L^2(h\dd x)}
 =\langle g,JPz\rangle_{L^2(\Om)}.
\]
Consequently
\[
 \langle g-Jf,z\rangle=\langle g,z-JPz\rangle,
 \qquad \|Pg\|_{L^2(h\dd x)}\le\|g\|_{L^2(\Om)}.
\]
The latter inequality follows from Cauchy--Schwarz on each section.

Set $L_n=p_n+2^n b_n$. Using \eqref{eq:transverseP} and
\eqref{eq:residual} in \eqref{eq:wweak}, we find
\begin{equation}\label{eq:weakestimate}
 |\mathfrak a_\Om(w,z)|
 \le L_n\eps\|g\|_{L^2(\Om)}
                   \|\nabla_yz\|_{L^2(\Om)}.
\end{equation}
In particular, $z=w$ gives
\begin{equation}\label{eq:energyError}
 \|\nabla w\|_{L^2(\Om)}
 \le L_n\eps\|g\|_{L^2(\Om)}.
\end{equation}

Since $v\in\mathcal V_h$
has zero weighted integral, we get $g-Jf$ is orthogonal to $Jv$
and \eqref{eq:residualOrthogonal} gives
\begin{equation}\label{eq:energyOrthogonal}
 \mathfrak a_\Om(w,Jv)=0.
\end{equation}
So, we obtain
\begin{align}
 \langle (T-S)g,g\rangle
 &=\langle w,g\rangle
   =\mathfrak a_\Om(w,u)\notag\\
 &=\mathfrak a_\Om(w,w+Jv)
   =\|\nabla w\|_{L^2(\Om)}^2.
   \label{eq:quadraticIdentity}
\end{align}
Thus the quadratic form of $T-S$ is nonnegative. By
\eqref{eq:energyError},
\[
 0\le\langle (T-S)g,g\rangle
 \le L_n^2\eps^2\|g\|_{L^2(\Om)}^2.
\]
Finally, for every $n\ge2$,
\[
 L_n=2^{3n/2}\bigl(n-1+2^{-n/2}+2^{-n}\bigr)\le4^n.
\]
The inequality is immediate for $n=2,3$; for $n\ge4$ it follows
from $2^{-n/2}+2^{-n}\le1$ and $n\le2^{n/2}$.
Consequently $L_n^2\le16^n$, which proves \eqref{eq:inverse}.
\end{proof}

\begin{remark}
The square in \eqref{eq:quadraticIdentity} is what turns the
$O(\eps)$ bound on the energy error into an $O(\eps^2)$ bound
on the inverse operators. The argument compares the entire
mean-zero operators, so it does not require identifying or
matching individual eigenfunctions.
\end{remark}

\section{Eigenvalue comparison and uniform dependence on the index}\label{sec:comparison}

\begin{proof}[Proof of Theorem~\ref{thm:main}]
The positive eigenvalues of $T$, arranged in nonincreasing order,
are $1/\mu_k(\Om)$, $k\ge1$. Since $J$ is an isometry and $P=J^*$,
$S$ is unitarily equivalent to $A_N^{-1}$ on $J\mathcal H_N$
and is zero on $\ker P$. Consequently its positive eigenvalues
are $1/\lambda_k$, with the same multiplicities as those of $A_N$.

The min--max principle applied to \eqref{eq:inverse} yields
\begin{equation}\label{eq:reciprocal}
 \frac1{\lambda_k}
 \le\frac1{\mu_k(\Om)}
 \le\frac1{\lambda_k}+16^n\eps^2,
 \qquad k\ge1.
\end{equation}
Inverting these positive quantities proves \eqref{eq:main}.
It also gives the slightly more precise estimate
\begin{equation}\label{eq:preciseGap}
 0\le\lambda_k-\mu_k(\Om)
 \le\frac{16^n\eps^2\lambda_k^2}
           {1+16^n\eps^2\lambda_k}
 \le16^n\eps^2\lambda_k^2.
\end{equation}

It remains to bound $\lambda_k$ using only $n$ and $k$.
Let $M=\sup_{0<x<1}h(x)$. Convexity and the endpoints imply
\begin{equation}\label{eq:tentPower}
 h(x)\ge M\min\{x,1-x\}^{m}.
\end{equation}
Indeed, for any $t\in(0,1)$, the homothetic inclusion toward $A$
gives
\[
 h(x)\ge (x/t)^m h(t)\ge x^m h(t),\qquad x\le t,
\]
and the inclusion toward $B$ gives
\[
 h(x)\ge\left(\frac{1-x}{1-t}\right)^m h(t)
       \ge(1-x)^m h(t),\qquad x\ge t.
\]
Both bounds imply $h(x)\ge d(x)^m h(t)$, so taking the supremum
in $t$ proves \eqref{eq:tentPower}. In particular,
\[
 4^{-m}M\le h(x)\le M,\qquad \frac14\le x\le\frac34.
\]
Partition $[1/4,3/4]$ into $k+1$ equal subintervals, each of
length $\ell=1/[2(k+1)]$. On each subinterval take its first
Dirichlet sine function and extend it by zero. These functions
belong to $\mathcal V_h$ and span a $(k+1)$-dimensional space.
Their supports are disjoint up to endpoints, so every nonzero
function $\phi$ in their span satisfies
\[
 \frac{\int_0^1h|\phi'|^2\dd x}{\int_0^1h|\phi|^2\dd x}
 \le4^m\frac{\pi^2}{\ell^2}
 =4^n\pi^2(k+1)^2.
\]
The min--max principle, with the zero eigenvalue included, gives
\begin{equation}\label{eq:eigenvalueUpper}
 \lambda_k\le4^n\pi^2(k+1)^2.
\end{equation}
Combining \eqref{eq:preciseGap} and \eqref{eq:eigenvalueUpper}
proves \eqref{eq:gap} with the constant in
\eqref{eq:constantnk}.
\end{proof}

\section{Sharpness of the quadratic estimate}\label{sec:rectsharp}

Now, we show that the exponent two in Theorem~\ref{thm:main} is optimal.

Let $0<\eps<1/\sqrt2$, put $a=\sqrt{1-\eps^2}$, and define
\[
 \Omega_\eps=(0,a)\times(0,\eps).
\]
Then $\diam\Omega_\eps=1$ and $W(\Omega_\eps)=\eps$,
where $W$ is the projective width defined in \eqref{eq:width}.
Indeed, the projection length onto a unit normal $(v_1,v_2)$ is
$a|v_1|+\eps|v_2|\geq\eps$, with equality for $(v_1,v_2)=(0,1)$.
Use the diameter joining $(0,0)$ to $(a,\eps)$ as the comparison
axis.

Separation of variables gives
\[
 \operatorname{spec}_N(\Omega_\eps)
 =\left\{\pi^2\left(\frac{p^2}{a^2}
                       +\frac{q^2}{\eps^2}\right):
              p,q\in\mathbb N_0\right\}.
\]
Consequently, for every fixed $k\ge1$ and
$\eps<1/\sqrt{k^2+1}$,
\begin{equation}\label{eq:rectmu}
 \mu_k(\Omega_\eps)=\frac{k^2\pi^2}{1-\eps^2}.
\end{equation}
The corresponding first $k$ positive modes have no transverse
oscillation.

Put $b=\eps^2$. Up to a constant factor, the section length is
\begin{equation}\label{eq:rectweight}
 \rho_b(x)=
 \begin{cases}
 x/b,&0<x<b,\\
 1,&b\le x\le1-b,\\
 (1-x)/b,&1-b<x<1.
 \end{cases}
\end{equation}
More precisely, the section length is $(\eps/a)\rho_b$. Its density is trapezoidal.

We solve
\[
 -(\rho_b u')'=q^2\rho_bu
\]
with natural endpoint conditions. On the left cap $(0,b)$, the
regular solution is $u(x)=J_0(qx)$. The singular Bessel solution
is excluded by finite weighted energy. At $x=b$,
\[
 \frac{u'(b)}{q u(b)}
 =-\frac{J_1(qb)}{J_0(qb)}.
\]
Thus, on the central interval, the solution can be written as a
multiple of
\[
 \cos\bigl(q(x-b)+\delta(qb)\bigr),\qquad
 \delta(z)=\arctan\frac{J_1(z)}{J_0(z)}.
\]
The right cap is the reflection of the left cap. Matching both
logarithmic derivatives gives, for every fixed low mode,
\begin{equation}\label{eq:rectroot}
 q_k(1-2b)+2\delta(q_kb)=k\pi,
 \qquad \mu_k(N_\eps)=q_k^2.
\end{equation}
For sufficiently small $b$, $J_0(q_kb)>0$ and the principal
arctangent is used. The cap solutions have no zeros, and the phase
on the central interval identifies the root with the $k$th
positive eigenvalue. For completeness, on any fixed bounded $q$ interval the phase
function on the left of \eqref{eq:rectroot} has derivative
$1-b+O(b^3)>0$. It vanishes at $q=0$ and converges uniformly to
$q$. Thus its successive positive roots correspond exactly to
$k=1,2,\ldots$ in the bounded spectral range under consideration.

Here $J_0$ and $J_1$ are specified by the convergent series
\[
 J_0(z)=\sum_{j=0}^{\infty}\frac{(-1)^j(z/2)^{2j}}{(j!)^2},
 \qquad
 J_1(z)=\sum_{j=0}^{\infty}
       \frac{(-1)^j(z/2)^{2j+1}}{j!(j+1)!}.
\]
Termwise differentiation gives $J_0'=-J_1$, and direct series
 division gives
\[
 \frac{J_1(z)}{J_0(z)}=\frac z2+\frac{z^3}{16}+O(z^5),
 \qquad
 \delta(z)=\frac z2+\frac{z^3}{48}+O(z^5).
\]
Therefore \eqref{eq:rectroot} becomes
\[
 q_k(1-b)+\frac{q_k^3b^3}{24}+O_k(b^5)=k\pi.
\]
The implicit function theorem at $(q,b)=(k\pi,0)$ gives
\begin{align}
 q_k&=\frac{k\pi}{1-b}
       -\frac{(k\pi)^3b^3}{24(1-b)^4}+O_k(b^5),\label{eq:rectq}\\
 \mu_k(N_\eps)
 &=\frac{k^2\pi^2}{(1-b)^2}
       -\frac{(k\pi)^4b^3}{12(1-b)^5}+O_k(b^5).
       \label{eq:rectlambda}
\end{align}
In particular, the less precise expansion
\[
 \mu_k(N_\eps)=\frac{k^2\pi^2}{(1-\eps^2)^2}+O_k(\eps^6)
\]
is already sufficient for sharpness.

\begin{proposition}\label{prop:rectangle}
For each fixed $k\ge1$, the diameter-one rectangles above satisfy
\begin{equation}\label{eq:rectgap}
 \boxed{\quad
 \mu_k(N_\eps)-\mu_k(\Omega_\eps)
 =k^2\pi^2\eps^2+2k^2\pi^2\eps^4+O_k(\eps^6).
 \quad}
\end{equation}
In particular,
\[
 \lim_{\eps\downarrow0}
 \frac{\mu_k(N_\eps)-\mu_k(\Omega_\eps)}{W(\Omega_\eps)^2}
 =k^2\pi^2>0.
\]
\end{proposition}

\begin{proof}
Subtract \eqref{eq:rectmu} from \eqref{eq:rectlambda}, using
$b=\eps^2$. The leading difference is exactly
\[
 k^2\pi^2\left(\frac1{(1-b)^2}-\frac1{1-b}\right)
 =\frac{k^2\pi^2b}{(1-b)^2}.
\]
Expansion proves \eqref{eq:rectgap}.
\end{proof}

\begin{corollary}
For every fixed $k\ge1$, an estimate
\[
 0\le\mu_k(N)-\mu_k(\Omega)\le C_k W(\Omega)^{2+\eta},
 \qquad \eta>0,
\]
cannot hold uniformly over diameter-one bounded convex planar domains.
Nor can the general upper error be replaced by a function
$o(W^2)$ as $W\to0$.
\end{corollary}

\begin{proof}
Apply the proposed estimate to the diameter-one rectangles in
Proposition~\ref{prop:rectangle}, divide by $W^2$, and let $W\to0$.
The left-hand quotient tends to $k^2\pi^2>0$, whereas the proposed
upper bound tends to zero.
\end{proof}

For an arbitrary diameter $D$, dilation of these same examples gives
\[
 \mu_k(N)-\mu_k(\Omega)
 =k^2\pi^2\frac{W^2}{D^4}
   +O_k\left(\frac{W^4}{D^6}\right),\qquad W/D\to0.
\]
This also verifies the scale of the quadratic estimate.

\end{document}